\RequirePackage{fix-cm}
\documentclass[11pt, a4paper, oneside, final, pagebackref]{amsart}
\usepackage[notref,notcite]{showkeys}
\usepackage[utf8]{inputenc}
\usepackage[T1]{fontenc}
\usepackage{amssymb}
\usepackage[stretch=10]{microtype}
\usepackage{mathtools}
\usepackage[DIV=11, headinclude=true]{typearea}
\usepackage[hidelinks, linktoc=all]{hyperref}

\providecommand{\href}[2]{#2}

\providecommand*{\backref}{}
\providecommand*{\backrefalt}{}
\renewcommand*{\backref}[1]{}
\renewcommand*{\backrefalt}[4]{%
	\ifcase #1 %
	\or
	  Cited page~#2.
	\else
	  Cited pages~#2.
	\fi
}

\newcommand{\Pbb}{\mathbb{P}}
\newcommand{\E}{\mathbb{E}}

\newcommand{\N}{\mathbb{N}}
\newcommand{\R}{\mathbb{R}}

\newcommand{\dd}{\mathop{}\!\mathrm{d}}

\DeclarePairedDelimiter{\abs}{\lvert}{\rvert}

\DeclarePairedDelimiter{\norm}{\lVert}{\rVert}

\DeclareMathOperator{\Id}{Id}

\renewcommand{\epsilon}{\varepsilon}

\renewcommand{\phi}{\varphi}
\renewcommand{\leq}{\leqslant}
\renewcommand{\geq}{\geqslant}

\newtheorem{thm}{Theorem}[section]

\newtheorem{definition}[thm]{Definition}
\newtheorem{lem}[thm]{Lemma}
\newtheorem{cor}[thm]{Corollary}
\newtheorem*{prop*}{Proposition}

\theoremstyle{definition}

\numberwithin{equation}{section}

\title[Mixing limit theorems]{Variations around Eagleson's Theorem
on mixing limit theorems for dynamical systems}
\author{Sébastien Gouëzel}

\address{Laboratoire Jean Leray, CNRS UMR 6629,
Université de Nantes, 2 rue de la
Houssinière,
44322 Nantes, France}
\email{sebastien.gouezel@univ-nantes.fr}

\begin{document}

\begin{abstract}
Eagleson's Theorem asserts that, given a probability-preserving map, if
renormalized Birkhoff sums of a function converge in distribution, then
they also converge with respect to any probability measure which is
absolutely continuous with respect to the invariant one. We prove a version
of this result for almost sure limit theorems, extending results of
Korepanov. We also prove a version of this result, in mixing systems, when
one imposes a conditioning both at time $0$ and at time $n$.
\end{abstract}

\maketitle

Let $T$ be an ergodic probability-preserving transformation on a probability
space $(X,m)$. Given a measurable function $f:X \to \R$, the question of the
convergence in distribution of renormalized Birkhoff sums $S_n f =
\sum_{k=0}^{n-1} f \circ T^k$ is central in ergodic theory. In physical
situations, where there is an a priori given reference probability measure
$P$ (for instance Lebesgue measure) which maybe differs from the invariant
measure, there can be a discussion of whether it is more natural to consider
such a distributional convergence with respect to the reference measure $P$
or to the invariant measure $m$. It turns out that this question is
irrelevant when $P$ is absolutely continuous with respect to $m$, by a
theorem of Eagleson~\cite{eagleson}: it is equivalent to have the
distributional convergence of $S_n f/B_n$ towards a limit $Z$ for $m$ or for
$P$, if $B_n \to \infty$. Since then, this theorem has proved extremely
useful, and has been extended to cover more general situations, see for
instance~\cite{aaronson_asymptotic, zweimuller_mixing}. In particular,
Eagleson's result holds in non-singular maps for processes which are
asymptotically invariant in probability.

Eagleson's Theorem has in particular been used to deduce limit theorems for a
map from limit theorems for an induced transformation. An important step in
this argument is to replace the invariant measure for the induced map (which
is the restriction of the invariant measure to the inducing set) by another
measure that takes into account the return time to the set, while keeping a
limit theorem, and this is proved using Eagleson's Theorem. Recently, a
similar inducing argument has been used by Melbourne and Nicol
in~\cite{melbourne_ASIP} to prove another kind of limit theorem, called
almost sure invariance principle and asserting that the Birkhoff sums can
almost surely be coupled with trajectories of a Brownian motion, so that the
mutual difference is suitably small. However, there was a difficulty in the
proof due to the lack of an analogue of Eagleson's result in this almost sure
setting. This gap has been fixed by Korepanov in~\cite{korepanov_ASIP} using
the specificities of the class of maps studied in~\cite{melbourne_ASIP}.

Our goal in this short note is to discuss two variations around Eagleson's
Theorem. First, in Section~\ref{sec:ASIP}, we give a general argument to show
that it is always equivalent to have an almost sure limit theorem for an
invariant probability measure or for an absolutely continuous one. Then, in
Section~\ref{sec:mixing}, we discuss distributional limit theorems for $S_n
f(x)/B_n$ when one conditions both on the positions of $x$ and $T^n x$ (where
conditioning only on $x$ corresponds to Eagleson's Theorem, and conditioning
only on $T^n x$ follows from Eagleson's Theorem applied to $T^{-1}$, but
conditioning simultaneously on both positions requires a new argument). Our
proofs in this note owe a lot to~\cite{zweimuller_mixing}
and~\cite{korepanov_ASIP}.

\section{Almost sure limit theorems}

\label{sec:ASIP}

In this section, we discuss a version of Eagleson's result that applies to
almost sure limit theorems. Given two probability measures $m_1$ and $m_2$,
the goal will be to construct a coupling between these two measures that
respects the orbit structure of the space, as in~\cite{korepanov_ASIP}. Then,
it will readily follow that an almost sure limit theorem with respect to
$m_1$ implies one with respect to $m_2$. Our argument works for general maps
but, contrary to~\cite{korepanov_ASIP}, our results are not quantitative. The
suitable definition of coupling we use is the following.

\begin{definition}
Let $(X, T)$ be a measurable map on a measurable space. A coupling along
orbits between two probability measures $m_1$ and $m_2$ (or more generally
between two finite measures of the same mass) is a measure $\rho$ on $X
\times X$ whose marginals are respectively $m_1$ and $m_2$, and such that,
for $\rho$ almost every $(x_1,x_2)$, there exist $n_1$ and $n_2$ with
$T^{n_1}x_1 = T^{n_2} x_2$. If there exists such a coupling, we say that
$m_1$ and $m_2$ can be coupled along orbits.
\end{definition}

Our goal is to show the following theorem:

\begin{thm}
\label{thm:coupling} Let $(X, T)$ be a measurable map on a standard
measurable space. Consider a $\sigma$-finite measure $\mu$ for which $T$ is
non-singular and ergodic. Let $m_1$ and $m_2$ be two probability measures
that are absolutely continuous with respect to $\mu$. Then they can be
coupled along orbits.
\end{thm}

Before proving the theorem, let us discuss the application to almost sure
limit theorems. An almost sure limit theorem with rate $r(n)$ between two
processes $(Z^1_n)_{n\in \N}$ and $(Z^2_n)_{n\in \N}$ defined on two
probability spaces $(\Omega_1, \Pbb_1)$ and $(\Omega_2, \Pbb_2)$ is a
coupling between these two processes, i.e., a measures $\Pbb$ on $\Omega_1
\times \Omega_2$ whose marginals are $\Pbb_1$ and $\Pbb_2$, such that for
$\Pbb$ almost every $\omega= (\omega_1,\omega_2)$, one has
$d(Z^1_n(\omega_1), Z^2_n(\omega_2)) = o(r(n))$. The most classical instance
of such a theorem is the almost sure invariance principle, asserting that the
Birkhoff sums $Z^1_n = S_n f$ can be coupled with the trajectories $Z_n^2$ of
a Brownian motion at integer times, where the error rate $r$ depends on the
problem under study.

\begin{cor}
Let $T$ be a probability-preserving ergodic map on a space $(X,m)$. Let $f: X
\to \R$ be measurable. Assume that the Birkhoff sums $S_n f$ satisfy an
almost sure limit theorem with rate $r$ for the measure $m$: they can be
coupled with a process $W_n$ such that, almost surely, $\abs{S_n f - W_n} =
o(r(n))$. Let $m'$ be a probability measure which is absolutely continuous
with respect to $m$. Assume moreover that, $m$-almost surely, $f(T^n x) =
o(r(n))$. Then $S_n f$ can also be coupled with $W_n$ for the measure $m'$,
with the same almost sure rate $r$.
\end{cor}
The growth assumption is for instance satisfied if $f$ is bounded and $r(n)$
tends to infinity, or if $f \in L^p$ and $r(n) = n^{1/p}$ (by Birkhoff
theorem applied to $\abs{f}^p$). These are the most typical situations in
applications.
\begin{proof}
It suffices to construct a coupling between $m$ and $m'$ such that, for
almost all $(x,y)$ for this coupling, one has $S_n f(x) - S_n f(y) =
o(r(n))$. We use the coupling along orbits given by
Theorem~\ref{thm:coupling}. In this case, almost every $(x,y)$ satisfy
$T^{k_1} x = T^{k_2} y$ for some $k_1, k_2$. Let $z=T^{k_1} x$. Let us prove
that, almost surely, $S_n f(x) = S_n f(z) + o(r(n))$ and $S_n f(y) = S_n f(z)
+ o(r(n))$, from which the result follows. It suffices to prove the first
estimate. For this, we note that $S_n f(z) - S_n f(x) = S_{k_1} f(T^n x) -
S_{k_1} f(x)$. The second term is constant, while the first one grows almost
surely at most like $o(r(n))$ under the assumptions of the corollary.
\end{proof}

Let us now turn to the proof of Theorem~\ref{thm:coupling}. Note first that,
if $m_1$ and $m_2$ can be coupled along orbits, as well as $m_2$ and $m_3$,
then it follows that $m_1$ and $m_3$ can also be coupled along orbits (this
follows from the composition of couplings theorem,
see~\cite[Lemma~A.1]{korepanov_ASIP}).

Note that there is no invariance assumption in the theorem for the measure
$\mu$, and that it does not have to be finite (although one can always assume
that $\mu$ is a probability measure, by replacing it with an equivalent
probability measure if necessary). However, in the applications we have in
mind, $\mu$ will typically be a probability measure, invariant under $T$. The
fact that the invariance is not relevant for this kind of theorem was pointed
out by Zweimüller in~\cite{zweimuller_mixing}: he was able to replace the use
of Birkhoff theorem by a variant which is valid without invariance, due to
Yosida. Denote by $\hat T : L^1(\mu) \to L^1(\mu)$ the transfer operator,
i.e., the predual of the composition by $T$ on $L^\infty$: it satisfies $\int
f\cdot g \circ T \dd\mu = \int \hat T f \cdot g \dd\mu$ for all $f\in
L^1(\mu)$ and $g\in L^\infty(\mu)$. Yosida's
Theorem~\cite[Theorem~2]{zweimuller_mixing} is the following:

\begin{thm}
Let $(X, T)$ be a measurable map on a measurable space. Consider a
$\sigma$-finite measure $\mu$ for which $T$ is non-singular and ergodic.
Then, for any $w \in L^1(\mu)$ with zero average, $\frac{1}{n}
\sum_{k=0}^{n-1} \hat T^k w$ tends to $0$ in $L^1(\mu)$.
\end{thm}

To prove Theorem~\ref{thm:coupling}, we will couple increasingly complicated
measures, relying ultimately on Yosida's Theorem. For starters, we begin with
a result that should be obvious.

\begin{lem}
\label{lem:coupling_Snf} Consider an integrable $f \geq 0$, and $n\geq 1$.
Then $f\dd\mu$ and $\hat T f \dd\mu$ can be coupled along orbits.
\end{lem}
\begin{proof}
While this looks obvious, it is enlightening to write down the details, to
understand what a coupling is. We let $\rho = (\Id, T)_* (f \dd\mu)$. The
first marginal of $\rho$ is $f \dd\mu$, while the second one is $T_*(f
\dd\mu) = \hat T f\dd\mu$, as desired.
\end{proof}
It follows from this lemma that $f\dd\mu$ and $\hat T^j f \dd\mu$ can be
coupled along orbits. Averaging, one gets the same result for $f\dd\mu$ and
$\frac{1}{n}\sum_{j=0}^{n-1}\hat T^j f \dd\mu$.

\begin{lem}
\label{lem:couple_part_mass} Consider two probability measures $m_1$ and
$m_2$ which are absolutely continuous with respect to $\mu$. Then there exist
two nonnegative measures $p_1 \leq m_1$ and $p_2 \leq m_2$, of mass $\geq
1/2$, that can be coupled along orbits.
\end{lem}
\begin{proof}
Denote by $f_1$ and $f_2$ the respective densities of $m_1$ and $m_2$ with
respect to $\mu$. Let $F_{i,n} = \frac{1}{n} \sum_{k=0}^{n-1} \hat T^k f_i$
for $i=1,2$ and $n > 0$. Let also $G_n(x) = \min(F_{1,n}(x), F_{2,n}(x))$. By
Yosida's Theorem, $\int \abs{F_{1,n} - F_{2,n}} \dd\mu$ tends to $0$. As
$G_n(x) = \frac{F_{1,n}(x)+F_{2,n}(x)- \abs{F_{1,n}(x) - F_{2,n}(x)}}{2}$, we
deduce that $\int G_n(x) \dd\mu \to 1$. In particular, we may choose $n$ such
that $\int G_n \dd\mu \geq 1/2$.

Consider a coupling along orbits $\rho$ between $m_1 = f_1 \dd\mu$ and $\nu_1
= F_{1,n} \dd\mu$, by Lemma~\ref{lem:coupling_Snf}. Define a new measure on
$X\times X$ by
\begin{equation*}
  \dd\tilde\rho(x,y) = \frac{G_n(y)}{F_{1,n}(y)} 1_{F_{1,n}(y)>0} \dd\rho.
\end{equation*}
As $G_n \leq F_{1,n}$ everywhere, one has $\tilde \rho \leq \rho$. The second
marginal of $\tilde \rho$ is the measure $G_n \dd\mu$ by construction, of
mass $\geq 1/2$. Hence, the first marginal of $\tilde\rho$ is a measure $p_1$
of mass at least $1/2$, dominated by the first marginal $m_1$ of $\rho$.
Moreover, by construction, $p_1$ is coupled along orbits with $G_n \dd\mu$.

In the same way, we obtain a measure $p_2 \leq m_2$ which is coupled along
orbits with $G_n \dd\mu$. Finally, $p_1$ and $p_2$ can be coupled along
orbits by transitivity. They satisfy the conclusion of the lemma.
\end{proof}

\begin{proof}[Proof of Theorem~\ref{thm:coupling}]
Start from two probability measures $m_1 = m_1^{(0)}$ and $m_2 = m_2^{(0)}$
that we want to couple along orbits. By Lemma~\ref{lem:couple_part_mass},
there exists a coupling $\rho_0$ along orbits between parts $p_1^{(0)}$ and
$p_2^{(0)}$ of mass at least $1/2$ of respectively $m_1^{(0)}$ and
$m_2^{(0)}$. Let $m_i^{(1)} = m_i^{(0)} - p_i^{(0)}$ be the uncoupled parts,
they have the same mass $\leq 1/2$. Applying Lemma~\ref{lem:couple_part_mass}
to these two measures, we obtain a coupling $\rho_1$ between parts
$p_i^{(1)}$ of these measures, leaving parts $m_i^{(2)}$ uncoupled, with mass
$\leq 1/4$. Iterate this process. Then $\rho = \sum \rho_i$ is the desired
coupling along orbits between $m_1$ and $m_2$.
\end{proof}

\section{Mixing transformations}

\label{sec:mixing}

Let $T$ be an ergodic map preserving a probability measure $m$. Eagleson's
Theorem ensures that, if $S_n f/B_n$ converges in distribution with respect
to $m$ towards a random variable $Z$, and $B_n\to \infty$, then this
convergence also holds with respect to any probability measure $m'$ which is
absolutely continuous with respect to $m$. We want to see what happens when
we condition on the position at two moments of time. A typical example is to
fix two sets $Y_1$ and $Y_2$ and only consider those trajectories that start
at time $0$ in $Y_1$ and end at time $n$ in $Y_2$. Conditioning at time $0$
is Eagleson's Theorem, conditioning at time $n$ follows from Eagleson's
Theorem applied in the natural extension and a change of variables, but the
simultaneous conditioning requires a new argument. When the map is mixing, we
prove that there is indeed such a limit theorem.

\begin{thm}
\label{thm:mixing_Eagleson} Let $T$ be an ergodic probability-preserving map
on $(X,m)$. Assume that $T$ is mixing. Let $f:X \to \R$ be a measurable
function such $S_n f/B_n$ converges in distribution to a real random variable
$Z$, where $B_n \to \infty$. Let $\phi_1, \phi_2 : X \to \R$ be two
nonnegative square integrable functions with $\int \phi_1 = \int \phi_2 = 1$.
Define a sequence of measures $m_n$ by $m_n(U) = \int_U \phi_1 \cdot \phi_2
\circ T^n \dd m$. They satisfy $m_n(X) \to 1$ by mixing. Then the random
variables $S_n f/B_n$ on the probability spaces $(X, m_n/m_n(X))$ converge in
distribution to $Z$.
\end{thm}

A sequence of real random variables $Z_n$ converges in distribution to $Z$ if
and only if, for any continuous bounded function $g$, then $\E(g(Z_n)) \to
\E(g(Z))$. By a density argument, it is even enough to restrict to bounded
Lipschitz functions. Writing an absolutely continuous probability measure
$m'$ as $\phi \dd m$ where $\phi$ is nonnegative and has integral $1$, then
Eagleson's Theorem can be restated as the convergence
  \begin{equation}
  \int g(S_n f/B_n) \phi \dd m \to \E(g(Z))) \int \phi \dd m.
  \end{equation}
By linearity, this even holds for any integrable function $\phi$.

What we have to do to prove Theorem~\ref{thm:mixing_Eagleson} is to prove the
same convergence, but when one multiplies by two functions $\phi_1$ and
$\phi_2 \circ T^n$, where $\phi_1,\phi_2 \in L^2$. More precisely, we have to
show that
\begin{equation}
\label{eq:mixing}
  \int \phi_1 \cdot g(S_n f/B_n) \cdot \phi_2\circ T^n \dd m\to \left(\int \phi_1 \dd m\right)
  \E(g(Z)) \left(\int \phi_2 \dd m\right).
\end{equation}

Without loss of generality, we can assume that $\phi_1$ and $\phi_2$ are
bounded, as a density argument readily gives the general conclusion. Let us
fix once and for all two such bounded functions. As in the discussion of
Eagleson's Theorem, we will in fact prove the convergence~\eqref{eq:mixing}
without assuming that the functions $\phi_1$ and $\phi_2$ are nonnegative,
although this condition is necessary for the probabilistic interpretation put
forward in the statement of Theorem~\ref{thm:mixing_Eagleson}. When $\phi_2$
is constant, the convergence~\eqref{eq:mixing} holds by Eagleson's Theorem.
Hence, we can without loss of generality replace $\phi_2$ with $\phi_2-\int
\phi_2 \dd m$, and assume that $\int \phi_2 \dd m= 0$.

The proof relies on the following lemma.
\begin{lem}
\label{lem:L2_bound} Assume that $T$ is mixing and $\phi_1,\phi_2$ are two
bounded functions with moreover $\int \phi_2 \dd m= 0$. Let $\epsilon>0$.
There exist $k$ and $N$ such that, for any $n\geq N$,
  \begin{equation}
  \norm*{ \frac{1}{k}\sum_{j=0}^{k-1} \phi_1 \circ T^j \cdot \phi_2 \circ T^{n+j} }_{L^2} \leq \epsilon.
  \end{equation}
\end{lem}
\begin{proof}
Let us expand the square:
  \begin{multline*}
  \int \left(\frac{1}{k}\sum_{j=0}^{k-1} \phi_1 \circ T^j \cdot \phi_2 \circ T^{n+j} \right)^2 \dd m
  \\= \frac{1}{k} \int \phi_1^2\cdot (\phi_2 \circ T^n)^2 \dd m
  + \frac{2}{k} \sum_{j=1}^k (1-j/k) \int \phi_1 \cdot \phi_1 \circ T^j \cdot (\phi_2 \cdot \phi_2 \circ T^j)\circ T^n \dd m.
  \end{multline*}
The first term is bounded by $C/k$ for $C = \norm{\phi_1}_{L^\infty}^2
\norm{\phi_2}_{L^\infty}^2$. When $n$ tends to infinity (and $k$ is fixed),
every integral in the second term tends to the product of the integrals, by
mixing. Hence, it is bounded by $2\norm{\phi_1}_{L^\infty}^2 \left| \int
\phi_2 \cdot \phi_2 \circ T^j \right|$ if $n$ is large enough. Choose $A$
such that this term is $\leq \epsilon$ for $j\geq A$ (again by mixing, and
using the fact that $\int \phi_2 = 0$). If $n$ is large enough, we obtain a
bound
  \begin{equation}
  \frac{C}{k} + \frac{2}{k}\sum_{j=1}^{A-1} C + \frac{2}{k}\sum_{j=A}^{k-1} \epsilon
  \leq (C+2AC)/k + 2\epsilon.
  \end{equation}
This concludes the proof, first by taking $k$ large enough but fixed so that
$(C+2AC)/k \leq \epsilon$, and then $n$ large enough so that the above mixing
argument applies.
\end{proof}

\begin{lem}
\label{lem:invariance} Assume that $B_n \to \infty$. We have
  \begin{equation}
   \int \phi_1 \cdot g(S_n f/B_n) \cdot \phi_2\circ T^n \dd m - \int \phi_1\circ T \cdot g(S_n f/B_n) \cdot \phi_2\circ T^{n+1} \dd m\to 0.
  \end{equation}
\end{lem}
\begin{proof}
As the measure is invariant, the difference between these two integrals is
equal to
  \begin{equation}
  \int \phi_1\circ T \cdot (g(S_n f\circ T/B_n) -g(S_nf/B_n)) \cdot \phi_2 \circ T^{n+1}.
  \end{equation}
Since $g$ is bounded and Lipschitz continuous, and $\phi_1$ and $\phi_2$ are
bounded, this is bounded by
\begin{multline*}
  C \int \min(1, \abs{S_n f \circ T - S_n f}/B_n) \dd m
  = C \int \min(1, \abs{f \circ T^n- f}/B_n) \dd m
  \\
  \leq C \int \min(1, \abs{f \circ T^n}/B_n) \dd m + C \int \min(1, \abs{f}/B_n) \dd m
  = 2C \int \min(1, \abs{f}/B_n) \dd m,
\end{multline*}
where we used the invariance of the measure for the last equality. This bound
tends to $0$ when $n$ tends to infinity, as $B_n \to \infty$.
\end{proof}

\begin{proof}[Proof of Theorem~\ref{thm:mixing_Eagleson}]
We prove the convergence~\eqref{eq:mixing} when $\phi_1$ and $\phi_2$ are
bounded and $\phi_2$ has zero average. Lemma~\ref{lem:invariance} (iterated
several times) ensures that, for any given $k$,
  \begin{equation}
  \int \phi_1 \cdot g(S_n f/B_n) \cdot \phi_2\circ T^n \dd m =
  \int \left( \frac{1}{k} \sum_{j=0}^{k-1} \phi_1 \circ T^j \cdot \phi_2 \circ T^{n+j} \right) g(S_n f /B_n) \dd m+o_n(1).
  \end{equation}
The integral on the right hand side is bounded by a constant multiple of the
$L^2$ norm of $\frac{1}{k} \sum_{j=0}^{k-1} \phi_1 \circ T^j \cdot \phi_2
\circ T^{n+j}$. If $k$ is fixed but large enough, this norm is bounded by
$\epsilon$ for large enough $n$, thanks to Lemma~\ref{lem:L2_bound}.
Therefore, $\int \phi_1 \cdot g(S_n f/B_n) \cdot \phi_2\circ T^n \dd m$ is
bounded in absolute value by $2\epsilon$. This concludes the proof.
\end{proof}

We can generalize the result as follows. Assume that $T$ is mixing of order
$p$. Let $F_n$ be a sequence of functions, taking values in a metric space
$M$, which is asymptotically invariant in the sense that $d(F_n, F_n \circ
T)$ tends to $0$ in probability, and such that $F_n$ converges in
distribution towards a random variable $Z$ on $M$. Then, for any bounded
functions $\phi_1,\dotsc,\phi_p$, for any $g:M\to \R$ Lipschitz and bounded,
  \begin{equation}
  \int \prod \phi_i \circ T^{n_i} \cdot g(F_n) \dd m
  \end{equation}
converges to $\prod (\int \phi_i) \cdot \E(g(Z))$, when $n$ and all the
$n_{i+1}-n_i$ tend to infinity. More formally, for any $\epsilon>0$, there
exists $N$ such that, for any $n$ and $n_1<\dotsb<n_p$ with $n\geq N$ and
$n_{i+1}-n_i \geq N$, then the above integral is within $\epsilon$ of $\prod
(\int \phi_i) \cdot \E(g(Z))$. This asserts that one can condition on the
position of the particle at $p$ times if these times are enough separated,
and still get the same limiting behavior.

The proof is the same as for Theorem~\ref{thm:mixing_Eagleson}. First, we use
order $p$ mixing to see that the sum $\frac{1}{k} \sum_{j=0}^{k-1} \prod
\phi_i \circ T^{n_i +j}$ is close in $L^2$ norm to $\prod (\int \phi_i)$ if
$k$ is large, and the $n_{i+1}-n_i$ are even larger. Then, one concludes
exactly as above.

\bibliography{biblio}
\bibliographystyle{amsalpha}

\end{document}